\def\marker{\>\hbox{${\vcenter{\vbox{
    \hrule height 0.4pt\hbox{\vrule width 0.4pt height 6pt
    \kern6pt\vrule width 0.4pt}\hrule height 0.4pt}}}$}\>}

\documentclass[12pt]{article}

\usepackage[left=2cm,right=2cm,top=3cm,bottom=3cm]{geometry}

\usepackage{amsmath, amssymb, latexsym, amsthm}
\usepackage{fullpage}

\usepackage{tikz}
\usetikzlibrary{shapes,decorations.pathreplacing,shapes.geometric}

\newtheorem{theorem}{Theorem} 
\newtheorem{theorem*}{Theorem} 

\newtheorem{lemma}[theorem]{Lemma}

\theoremstyle{definition}

\newtheorem{question}{Question}

\theoremstyle{remark}

{\end{enumerate}}

\newcommand{\RR}{\mathbb{R}}

\newcommand{\ee}{\mathbf{e}}
\newcommand{\ds}{\displaystyle }

\newcommand{\CL}[1]{\left\lceil #1 \right\rceil}

\DeclareMathOperator{\sgn}{sgn}

\title{Unit Hypercube Visibility Numbers of Trees}
\author{Eric Peterson\footnote{{Dept.\ of Mathematics, Univ.\ of Rhode Island, Kingston, RI; {\tt epeterson11492@my.uri.edu}.}}\: and Paul S.\ Wenger\footnote{{School of Mathematical Sciences, Rochester Institute of Technology, Rochester, NY; {\tt pswsma@rit.edu}}}}

\begin{document}

\maketitle

\begin{abstract}
A visibility representation of a graph $G$ is an assignment of the vertices of $G$ to geometric objects such that vertices are adjacent if and only if their corresponding objects are ``visible" each other, that is, there is an uninterrupted channel, usually axis-aligned, between them.
Depending on the objects and definition of visibility used, not all graphs are visibility graphs.
In such situations, one may be able to obtain a visibility representation of a graph $G$ by allowing vertices to be assigned to more than one object.
The {\it visibility number} of a graph $G$ is the minimum $t$ such that $G$ has a representation in which each vertex is assigned to at most $t$ objects.

In this paper, we explore visibility numbers of trees when the vertices are assigned to unit hypercubes in $\RR^n$.
We use two different models of visibility: when lines of sight can be parallel to any standard basis vector of $\RR^n$, and when lines of sight are only parallel to the $n$th standard basis vector in $\RR^n$.
We establish relationships between these visibility models and their connection to trees with certain cubicity values.

{\bf Keywords: 05C62, visibility, cubicity}

\end{abstract}

\baselineskip18pt

\section{Introduction}

Broadly speaking, a visibility representation of a graph $G$ is an assignment of the vertices of $G$ to geometric objects embedded in an ambient space so that two vertices are adjacent if and only if there is a line of sight between their objects that intersects none of the other objects.
Visibility representations have been studied with a wide variety of geometric objects including bars~\cite{DHLM,TandT,Wis}, semi-bars~\cite{CCHMM}, rectangles~\cite{OnRVG,DH1,DH2}, points~\cite{Point}, and circular arcs~\cite{Arc,H2}.

Historically, the study of visibility representations was motivated by VLSI design.
Reflecting these roots in circuit design, it may be inappropriate to have vertices assigned to objects whose sizes can differ by arbitrary amounts, since the components of electronic circuits are roughly uniform in size.
A natural restriction is to require that all vertices are assigned to objects of the same size.
Such representations have been studied with bars~\cite{DGH,UnitBar} and rectangles~\cite{UnitRec}.

In this paper we study visibility representations in which vertices are assigned to unit hypercubes in $\RR^n$.
Throughout we will use the standard basis vectors $\ee_1,\ldots,\ee_n$ as  a basis for $\RR^n$.
We consider two different versions of visibility, which are inspired by unit bar visibility graphs and unit rectangle visibility graphs.
In a {\it unit bar visibility representation} of a graph, vertices are assigned to disjoint horizontal bars of length 1 in the plane, and two bars see each other if there is a vertical channel of positive width joining them that intersects no other bar (see Figure~\ref{fig:barvis}).
Unit bar visibility graphs were first studied by Dean and Veytsel~\cite{UnitBar}.
In a {\it unit rectangle visibility representation} of a graph, vertices are assigned to disjoint axis-aligned unit rectangles in the plane, and two rectangles see each other if there is a vertical or horizontal channel of positive width joining them that intersects no other rectangle (see Figure~\ref{fig:barvis}).
Unit rectangle visibility graphs were first studied by Dean, Ellis-Monaghan, Hamilton, and Pangborne~\cite{UnitRec}.

\begin{figure}
\centering
\begin{tikzpicture}

\draw[line width=1.5] (.25,0) -- (1.25,0);
\draw[line width=1.5] (1.5,0) -- (2.5,0);
\draw[line width=1.5] (.75,1) -- (1.75,1);
\draw[line width=1.5] (1,2) -- (2,2);
\draw[dashed, line width=1.5] (1.5,1) -- (1.5,2);
\draw[dashed, line width=1.5] (1.9,0) -- (1.9,2);
\draw[dashed, line width=1.5] (.875,0) -- (.875,1);
\draw[dashed, line width=1.5] (1.625,0) -- (1.625,1);
\node at (1.25,1.3) {$v_2$};
\node at (1.75,2.3) {$v_3$};
\node at (.75,-.4) {$v_1$};
\node at (2,-.4) {$v_4$};

\draw[fill] (5,1) circle (4pt);
\node at (5,1.3) {$v_1$};
\draw[fill] (6.5,1) circle (4pt);
\node at (6.5,1.3) {$v_2$};
\draw[fill] (8,2) circle (4pt);
\node at (8,2.3) {$v_3$};
\draw[fill] (8,0) circle (4pt);
\node at (8,-.4) {$v_4$};
\draw[line width=1.5] (5,1) -- (6.5,1) -- (8,2) -- (8,0) -- (6.5,1);

\draw[line width=1.5] (10,.6)--(10,1.4)--(10.8,1.4)--(10.8,.6)--(10,.6);

\draw[shift={(1.2,0)},line width=1.5] (10,.6)--(10,1.4)--(10.8,1.4)--(10.8,.6)--(10,.6);

\draw[shift={(2.4,.65)},line width=1.5] (10,.6)--(10,1.4)--(10.8,1.4)--(10.8,.6)--(10,.6);

\draw[shift={(2.4,-.65)},line width=1.5] (10,.6)--(10,1.4)--(10.8,1.4)--(10.8,.6)--(10,.6);

\node at (10.4,1) {$v_1$};
\node at (11.6,1) {$v_2$};
\node at (12.8,1.65) {$v_3$};
\node at (12.8,.35) {$v_4$};

\draw[dashed, line width=1.5] (10.8,1) -- (11.2,1);
\draw[dashed, line width=1.5] (12,1.3) -- (12.4,1.3);
\draw[dashed, line width=1.5] (12,.7) -- (12.4,.7);
\draw[dashed, line width=1.5] (12.8,1.25) -- (12.8,.75);

\end{tikzpicture}
\caption{A unit bar visibility representation and a unit rectangle visibility representation of a graph.}\label{fig:barvis}
\end{figure}
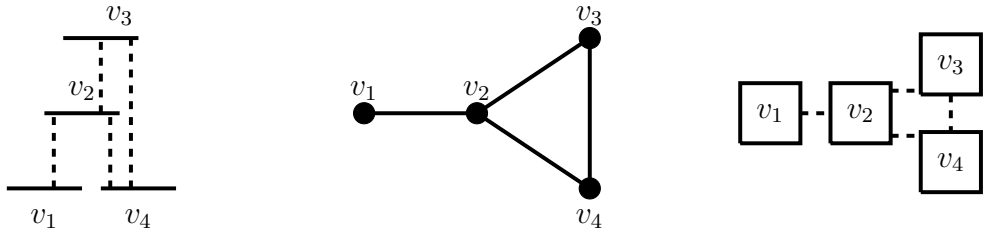

There is an important distinction between the models of unit bar and unit rectangle visibility graphs: in a unit bar visibility graph, the lines of sight are orthogonal to the affine spaces defined by the objects; in a unit rectangle visibility graph, the lines of sight are orthogonal to faces of the rectangles, but live in the same ambient space as the objects.
There is also a clear connection between unit bar and unit rectangle visibility representations: the horizontal and vertical lines of sight in a unit rectangle visibility representation correspond directly to unit bar visibility representations.

To elucidate this distinction, we define two versions of visibility for unit hypercube visibility graphs.
Throughout the paper, we use {\it $n$-cube} to mean an $n$-dimensional unit hypercube.
A graph $G$ has an {\it $n$-cube visibility representation} if the vertices of $G$ can be assigned to disjoint axis-aligned unit $n$-cubes in $\RR^{n}$ such that two vertices are adjacent if and only if there is an uninterrupted axis-aligned cylindrical channel of positive diameter between their respective $n$-cubes.
A graph with an $n$-cube visibility representation is an {\it $n$-cube visibility graph}. 
In this setting, a unit rectangle visibility graph is a $2$-cube visibility graph.
A graph $G$ has an {\it $n$-cube orthogonal visibility representation} if the vertices of $G$ can be assigned to disjoint unit $n$-cubes in $\RR^{n+1}$ such that the cubes are aligned with the first $n$ axes, and two vertices are adjacent if and only if there is an uninterrupted cylindrical channel of positive diameter that is parallel to $\ee_{n+1}$ between their respective $n$-cubes.
A graph with an $n$-cube orthogonal visibility representation is an {\it $n$-cube orthogonal visibility graph}. 
In this setting, a unit bar visibility graph is a $1$-cube orthogonal visibility graph.
See Figure~\ref{fig:orthandreg} for examples of a $2$-cube visibility representation and a $2$-cube orthogonal visibility representation.

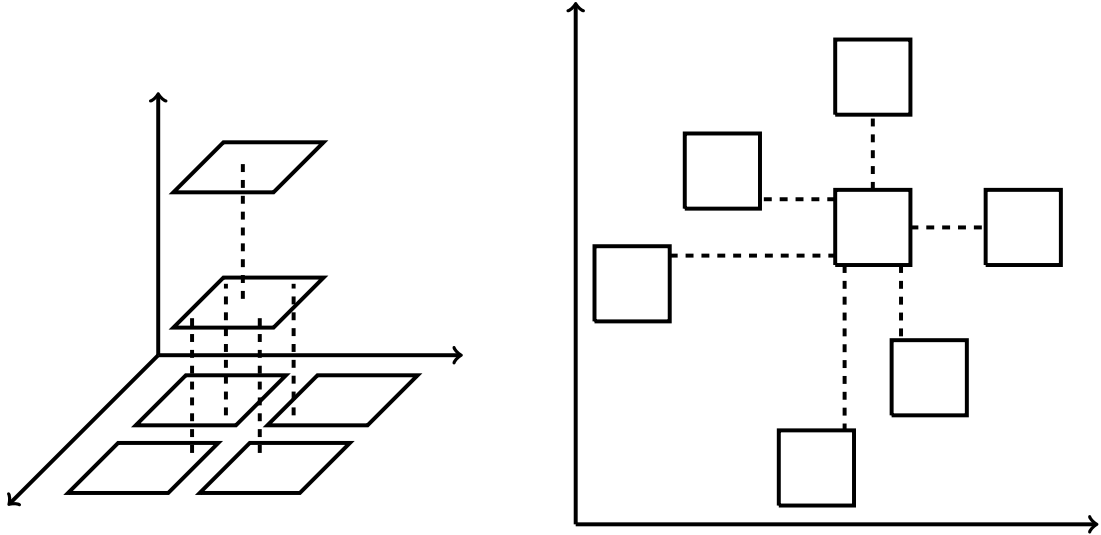
\begin{figure}
\centering
\begin{tikzpicture}

\node[trapezium, draw, trapezium left angle=45, trapezium right angle=135, minimum width = 2cm, line width = 1.5] at (0,0) {};

\node[trapezium, draw, trapezium left angle=45, trapezium right angle=135, minimum width = 2cm, line width = 1.5] at (1.75,0) {};

\node[trapezium, draw, trapezium left angle=45, trapezium right angle=135, minimum width = 2cm, line width = 1.5] at (.9,.9) {};

\node[trapezium, draw, trapezium left angle=45, trapezium right angle=135, minimum width = 2cm, line width = 1.5] at (2.65,.9) {};

\node[trapezium, draw, trapezium left angle=45, trapezium right angle=135, minimum width = 2cm, line width = 1.5] at (1.4,2.2) {};

\node[trapezium, draw, trapezium left angle=45, trapezium right angle=135, minimum width = 2cm, line width = 1.5] at (1.4,4) {};

\draw[dashed, line width=1.5] (.65,.2) -- (.65,2);
\draw[dashed, line width=1.5] (1.55,.2) -- (1.55,2);
\draw[dashed, line width=1.5] (1.1,.7) -- (1.1,2.45);
\draw[dashed, line width=1.5] (2,.7) -- (2,2.45);
\draw[dashed, line width=1.5] (1.325,2.25) -- (1.325,4.05);

\draw[->, line width=1.5] (.2,1.5) -- (4.25,1.5);
\draw[->, line width=1.5] (.2,1.5) -- (-1.8,-.5);
\draw[->, line width=1.5] (.2,1.5) -- (.2,5);

\draw [shift={(1.2,1.2)},line width=1.5] (8,1.5)--(9,1.5)--(9,2.5)--(8,2.5)--(8,1.5);
\draw [shift={(3.2,1.2)}, line width=1.5] (8,1.5)--(9,1.5)--(9,2.5)--(8,2.5)--(8,1.5);
\draw [shift={(-.8,1.95)}, line width=1.5] (8,1.5)--(9,1.5)--(9,2.5)--(8,2.5)--(8,1.5);
\draw [shift={(-2,.45)}, line width=1.5] (8,1.5)--(9,1.5)--(9,2.5)--(8,2.5)--(8,1.5);
\draw [shift={(1.2,3.2)}, line width=1.5] (8,1.5)--(9,1.5)--(9,2.5)--(8,2.5)--(8,1.5);
\draw [shift={(.45,-2)}, line width=1.5] (8,1.5)--(9,1.5)--(9,2.5)--(8,2.5)--(8,1.5);
\draw [shift={(1.95,-.8)}, line width=1.5] (8,1.5)--(9,1.5)--(9,2.5)--(8,2.5)--(8,1.5);

\draw[dashed, line width=1.5] (9.325,2.7) -- (9.325,.5);
\draw[dashed, line width=1.5] (10.075,2.7) -- (10.075,1.7);
\draw[dashed, line width=1.5] (9.7,3.7) -- (9.7,4.7);
\draw[dashed, line width=1.5] (9.2,3.575) -- (8.2,3.575);
\draw[dashed, line width=1.5] (7,2.825) -- (9.2,2.825);
\draw[dashed, line width=1.5] (10.2,3.2) -- (11.2,3.2);

\draw[->, line width=1.5] (5.75,-.75) -- (12.7,-.75);
\draw[->, line width=1.5] (5.75,-.75) -- (5.75,6.2);

\end{tikzpicture}

\caption{A unit $2$-cube orthogonal visibility representation of $K_{1,5}$ and a unit $2$-cube visibility representation of $K_{1,6}$.}\label{fig:orthandreg}
\end{figure}

The family of graphs that have a specified type of visibility representation is generally quite limited.
For instance, when objects are embedded in the plane and the directions of lines of sight are limited, there usually is a bound on the thickness of such a graph.
To address such limitations, one can assign vertices to more than one object in a visibility representation.
This approach was first taken by Chang, Hutchinson, Jacobson, Lehel, and West~\cite{BarVisNum} in the context of bar visibility representations.
The {\it bar visibility number} of a graph $G$ is the minimum $t$ such that $G$ has a bar orthogonal visibility representation in which each vertex is assigned to at most $t$ bars.
Bar visibility numbers were further studied by Axenovich, Beveridge, Hutchinson, and West~\cite{VisDirect} for directed graphs (edges are oriented towards whichever bar has a larger $y$-coordinate), and also by Gaub, Rose, and the second author~\cite{UnitBarVisNum} when all bars have unit length.

In this paper, we study visibility numbers of trees when vertices are assigned to sets of unit hypercubes.
The {\it $n$-cube visibility number} of a graph $G$, denoted $h^{(n)}(G)$ is the minimum $t$ such that there is a $n$-cube visibility representation of $G$ in which each vertex is assigned to at most $t$ unit $n$-cubes in $\RR^n$.
The {\it $n$-cube orthogonal visibility number} of a graph $G$, denoted $h^{(n)\perp}(G)$ is the minimum $t$ such that there is an $n$-cube orthogonal visibility representation of $G$ in which each vertex is assigned to at most $t$ $n$-cubes in $\RR^{n+1}$ (recall that these cubes are orthogonal to $\ee_{n+1}$).
See Figures~\ref{fig:multicube1} and~\ref{fig:multicube2} for examples.

In all $n$-cube visibility representations and $n$-cube orthogonal visibility representations we will add the additional requirement that no two cubes that are assigned to the same vertex can see each other.
Such lines of sight would allow us to use multiple $n$-cubes to act like an $n$-dimensional box with non-unit dimensions.

\begin{figure}
\centering
\begin{tikzpicture}

\draw [shift={(1.2,1.2)},line width=1.5] (8,1.5)--(9,1.5)--(9,2.5)--(8,2.5)--(8,1.5);
\draw [shift={(3.2,1.2)}, line width=1.5] (8,1.5)--(9,1.5)--(9,2.5)--(8,2.5)--(8,1.5);
\draw [shift={(-.8,1.95)}, line width=1.5] (8,1.5)--(9,1.5)--(9,2.5)--(8,2.5)--(8,1.5);
\draw [shift={(-2,.45)}, line width=1.5] (8,1.5)--(9,1.5)--(9,2.5)--(8,2.5)--(8,1.5);
\draw [shift={(1.2,3.2)}, line width=1.5] (8,1.5)--(9,1.5)--(9,2.5)--(8,2.5)--(8,1.5);
\draw [shift={(.45,-2)}, line width=1.5] (8,1.5)--(9,1.5)--(9,2.5)--(8,2.5)--(8,1.5);
\draw [shift={(1.95,-.8)}, line width=1.5] (8,1.5)--(9,1.5)--(9,2.5)--(8,2.5)--(8,1.5);

\draw[dashed, line width=1.5] (9.325,2.7) -- (9.325,.5);
\draw[dashed, line width=1.5] (10.075,2.7) -- (10.075,1.7);
\draw[dashed, line width=1.5] (9.7,3.7) -- (9.7,4.7);
\draw[dashed, line width=1.5] (9.2,3.575) -- (8.2,3.575);
\draw[dashed, line width=1.5] (7,2.825) -- (9.2,2.825);
\draw[dashed, line width=1.5] (10.2,3.2) -- (11.2,3.2);

\node[shift={(1.2,1.2)}] at (8.5,2) {$x$};
\node[shift={(1.2,1.2)}] at (10.5,2) {$v_6$};
\node[shift={(0,1.2)}] at (6.5,1.25) {$v_3$};
\node[shift={(1.2,1.2)}] at (6.5,2.75) {$v_2$};
\node[shift={(1.2,1.2)}] at (8.5,4) {$v_1$};
\node[shift={(1.2,0)}] at (7.75,0) {$v_4$};
\node[shift={(1.2,1.2)}] at (9.25,0) {$v_5$};


\draw [shift={(6.7,6.7)}, line width=1.5] (8,1.5)--(9,1.5)--(9,2.5)--(8,2.5)--(8,1.5);
\draw [shift={(4.7,6.7)}, line width=1.5] (8,1.5)--(9,1.5)--(9,2.5)--(8,2.5)--(8,1.5);
\draw [shift={(6.7,4.7)}, line width=1.5] (8,1.5)--(9,1.5)--(9,2.5)--(8,2.5)--(8,1.5);

\node at (15.2,8.7) {$x$};
\node at (15.2,6.7) {$v_8$};
\node at (13.2,8.7) {$v_7$};

\draw[shift={(6.7,6.7)}, dashed, line width=1.5] (8.5,1.5) -- (8.5,.5);
\draw[shift={(6.7,6.7)}, dashed, line width=1.5] (8,2) -- (7,2);
\end{tikzpicture}

\caption{A $2$-cube visibility representation of $K_{1,8}$ using at most two cubes per vertex.
The representation has two components.}\label{fig:multicube1}

\end{figure}
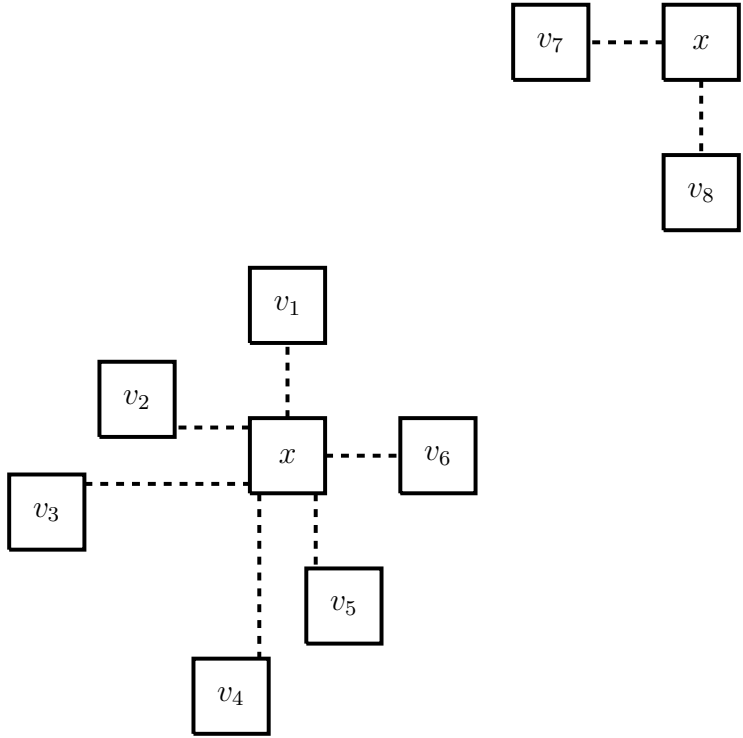

\begin{figure}
\centering
\begin{tikzpicture}
\node[trapezium, draw, trapezium left angle=45, trapezium right angle=135, minimum width = 2cm, line width = 1.5] at (0,0) {$v_1$};

\node[trapezium, draw, trapezium left angle=45, trapezium right angle=135, minimum width = 2cm, line width = 1.5] at (1.75,0) {$v_2$};

\node[trapezium, draw, trapezium left angle=45, trapezium right angle=135, minimum width = 2cm, line width = 1.5] at (.9,.9) {$v_3$};

\node[trapezium, draw, trapezium left angle=45, trapezium right angle=135, minimum width = 2cm, line width = 1.5] at (2.65,.9) {$v_4$};

\node[trapezium, draw, trapezium left angle=45, trapezium right angle=135, minimum width = 2cm, line width = 1.5] at (1.4,2.2) {$x$};

\node[trapezium, draw, trapezium left angle=45, trapezium right angle=135, minimum width = 2cm, line width = 1.5] at (1.4,4) {$v_5$};

\draw[dashed, line width=1.5] (.65,.2) -- (.65,2);
\draw[dashed, line width=1.5] (1.55,.2) -- (1.55,2);
\draw[dashed, line width=1.5] (1.1,.7) -- (1.1,2.45);
\draw[dashed, line width=1.5] (2,.7) -- (2,2.45);
\draw[dashed, line width=1.5] (1.325,2.4) -- (1.325,3.9);

\node[shift={(5,0)}, trapezium, draw, trapezium left angle=45, trapezium right angle=135, minimum width = 2cm, line width = 1.5] at (0,0) {$v_6$};

\node[shift={(5,0)}, trapezium, draw, trapezium left angle=45, trapezium right angle=135, minimum width = 2cm, line width = 1.5] at (1.75,0) {$v_7$};

\node[shift={(5,0)}, trapezium, draw, trapezium left angle=45, trapezium right angle=135, minimum width = 2cm, line width = 1.5] at (.9,.9) {$v_8$};

\node[shift={(5,0)}, trapezium, draw, trapezium left angle=45, trapezium right angle=135, minimum width = 2cm, line width = 1.5] at (2.65,.9) {$v_9$};

\node[shift={(5,0)}, trapezium, draw, trapezium left angle=45, trapezium right angle=135, minimum width = 2cm, line width = 1.5] at (1.4,2.2) {$x$};

\node[shift={(5,0)}, trapezium, draw, trapezium left angle=45, trapezium right angle=135, minimum width = 2cm, line width = 1.5] at (1.4,4) {$v_{10}$};

\draw[shift={(5,0)}, dashed, line width=1.5] (.65,.2) -- (.65,2);
\draw[shift={(5,0)}, dashed, line width=1.5] (1.55,.2) -- (1.55,2);
\draw[shift={(5,0)}, dashed, line width=1.5] (1.1,.7) -- (1.1,2.45);
\draw[shift={(5,0)}, dashed, line width=1.5] (2,.7) -- (2,2.45);
\draw[shift={(5,0)}, dashed, line width=1.5] (1.325,2.4) -- (1.325,3.9);

\end{tikzpicture}
\caption{A $2$-cube orthogonal visibility representation of $K_{1,10}$ using at most two cubes per vertex.
The representation has two components.}\label{fig:multicube2}
\end{figure}
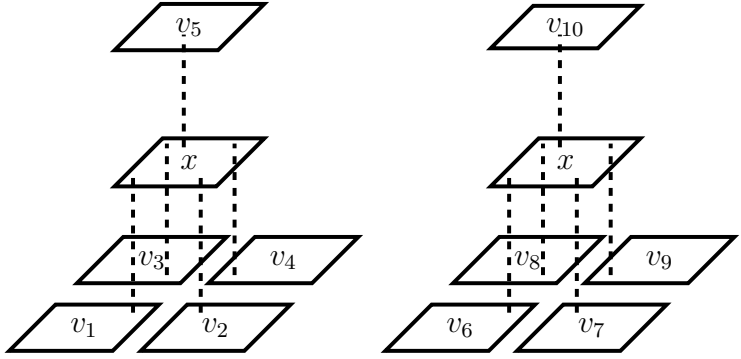

It is clear that every tree has bar visibility number $1$ (here the bars can have different length).
In~\cite{UnitBarVisNum}, Gaub et al.\ presented a linear-time algorithm that determines the $1$-cube orthogonal visibility number of any tree.
In this paper, we study the relation between $n$-cube visibility numbers and $(n-1)$-cube orthogonal visibility numbers of trees, and also their connections to related arboricity parameters.
We also characterize trees with $n$-cube orthogonal visibility representations in terms of trees that have representations as intersection graphs of $n$-cubes in $\RR^n$.

\section{Main results}

A tree is an {\it $n$-cube (orthogonal) visibility tree} if it is an $n$-cube (orthogonal) visibility graph.
A graph is an {\it $n$-cube (orthogonal) visibility forest} if it is the disjoint union of $n$-cube (orthogonal) visibility trees.
Let $\Upsilon_{h^{(n)}}(G)$ denote the minimum $k$ such that $G$ is the union of $k$ unit $n$-cube visibility forests.
Let $\Upsilon_{h^{(n)\perp}}(G)$ denote the minimum $k$ such that $G$ is the union of $k$ unit $n$-cube orthogonal visibility forests.
We call $\Upsilon_{h^{(n)}}(G)$  and $\Upsilon_{h^{(n)\perp}}(G)$ the {\it $n$-cube arboricity} and the {\it $n$-cube orthogonal arboricity} of $G$, respectively.

We begin by showing that the $n$-cube (orthogonal) visibility number of a tree $T$ is equal to the $n$-cube (orthogonal) arboricity of $T$.
Since the same proof works for both standard and orthogonal visibility models, we use $h^{(n)*}(T)$ and $\Upsilon_{{h}^{(n)*}}(T)$ to indicate the visibility number and arboricity, where the visibility may be standard or orthogonal.

Let $R$ be an $n$-cube (orthogonal) visibility representation of a graph $G$.
A \textit{component} of $R$ is a maximal set of $n$-cubes that are connected by the visibility relation.
An important aspect of the visibility models that we study is that it is possible for a visibility representation to have multiple components.
This is not true in some visibility models, notably the point visibility graphs studied in~\cite{Point}.
Let $G$ and $G'$ be graphs with visibility representations $R$ and $R'$ respectively ($R$ and $R'$ are assumed to be representations of the same type).
We can obtain a visibility representation of $G\cup G'$ by adding the cubes of $R'$ to $R$, and translating the cubes of $R'$ all by the same vector so that no cube from $R$ sees a cube of $R'$; we call the resulting representation a {\it disjoint union} of $R$ and $R'$.
The representations in Figures~\ref{fig:multicube1} and~\ref{fig:multicube2} can both be thought of as the disjoint union of two components.

We require the following technical lemma, a version of which appeared for $n=1$ as Lemma 1 in~\cite{UnitBarVisNum}.
Since the proof is essentially identical to the proof in~\cite{UnitBarVisNum}, we omit it here.

\begin{lemma}\label{lem:coords}
If $G$ is a unit $n$-cube orthogonal visibility graph, then there is a unit $n$-cube orthogonal visibility representation of $G$ in which all $n$-cubes have distinct coordinates in the direction of $\ee_{n+1}$.
\end{lemma}

For ease of discussion, we will refer to the $(n+1)$st coordinate of an $n$-cube in an $n$-cube orthogonal visibility representation as its {\it height}.

\begin{theorem}\label{vis=arb}
If $T$ is a tree, then $\Upsilon_{{h}^{(n)*}}(T) = h^{(n)*}(T)$.
\end{theorem}

\begin{proof} If $\Upsilon_{{h}^{(n)*}}(T) = k$, then there exists a decomposition of $T$ into $k$ $n$-cube (orthogonal) visibility forests.
A disjoint union of the $n$-cube (orthogonal) visibility representations of each of these forests is an $n$-cube (orthogonal) visibility representation of $T$ in which each vertex is assigned to at most $k$ cubes, so $h^{(n)*}(T)\leq \Upsilon_{{h}^{(n)*}}(T)$.

Let $h^{(n)*}(T) = t$ and let $R$ be an $n$-cube (orthogonal) visibility representation of $T$ in which each vertex is assigned to at most $t$ cubes.
If $R$ is an $n$-cube orthogonal visibility representation, then assume by Lemma~\ref{lem:coords} that all $n$-cubes have distinct heights.
Furthermore, assume that $R$ is chosen so that it contains the minimum number of pairs of $n$-cubes that correspond to the same vertex and lie in the same component of $R$.

Assume that there exists a $v \in V(T)$ such that two $n$-cubes corresponding to $v$ lie in the same component of $R$; call the component $R_1$.
For each edge $xy$ represented as a line of sight in $R_1$, choose a line of sight between cubes for $x$ and $y$ in $R_1$ of minimum Euclidean length, breaking ties arbitrarily.
The chosen lines of sight create a spanning forest $F$ of the $n$-cubes in $R_1$.
Because $T$ is a tree, there is at most one cube assigned to each vertex in each component of $F$.
Let $R'$ be a representation obtained by taking the disjoint union of $R-R_1$ and each of the visibility representations of the components of $F$.
Observe that $R'$ retains all edges represented in $R$ and contains fewer pairs of hypercubes that are assigned to the same vertex and lie in the same component.

It remains to show that $R'$ does not have any lines of sight that correspond to edges that are not in $T$.
Suppose that the $n$-cubes $B(a)$ and $B(b)$ correspond to vertices $a$ and $b$ that are not adjacent in $T$, but $B(a)$ and $B(b)$ see each other in $R'$; clearly $B(a)$ and $B(b)$ lie in $R_1$ in the representation $R$.
Thus, there is a channel of visibility between $B(a)$ and $B(b)$ in $R'$ that does not exist in $R_1$.
Therefore, there is a collection of $n$-cubes in $R_1$ that blocks the channel of visibility between $B(a)$ and $B(b)$, and these blocking cubes are in a different component of the spanning forest of $R_1$.
Label the blocking cubes $B(c_1),B(c_2),...,B(c_k)$ where $B(a)$ sees $B(c_1)$, $B(c_i)$ sees $B(c_{i+1})$ for $i\in[k-1]$, and $B(c_k)$ sees $B(b)$ (see Figure~\ref{blocker}).

Note that $ac_1\ldots c_kb$ is a walk in $T$ between $a$ and $b$; call the walk $W$.
Let the length of the channel between $B(a)$ and $B(b)$ be $d$.
The sum of the Euclidean lengths of the lines of sight corresponding to edges in $W$ is $d-k$ if $R$ is an $n$-cube visibility representation and $d$ if $R$ is an $n$-cube orthogonal visibility representation.

\begin{figure}
\begin{center}
\begin{tikzpicture}
\draw (0,0) -- (0,1) -- (1,1) -- (1,0) -- cycle;
\draw (1.5,0.1) -- (1.5,1.1) -- (2.5,1.1) -- (2.5,0.1) -- cycle;
\draw (3,0.5) -- (3,1.5) -- (4,1.5) -- (4,0.5) -- cycle;
\draw (3.2,0.3) -- (3.2,-0.7) -- (4.2,-0.7) -- (4.2,0.3) -- cycle;
\draw (4.4,0.1) -- (4.4,1.1) -- (5.4,1.1) -- (5.4,0.1) -- cycle;
\draw (7.6,1.1) -- (7.6,0.1) -- (8.6,0.1) -- (8.6,1.1) -- cycle;
\draw (9,0.1) -- (9,1.1) -- (10,1.1) -- (10,0.1) -- cycle;

\node[scale=0.8] at (0.5,0.5) {$B(a)$};
\node[scale=0.8] at (9.5,0.6) {$B(b)$};
\node[scale=0.8] at (2,0.6) {$B(c_1)$};
\node[scale=0.8] at (3.5,1) {$B(c_2)$};
\node[scale=0.8] at (4.9,0.6) {$B(c_3)$};
\node[scale=0.8] at (8.1,0.6) {$B(c_k)$};
\node[scale=0.8] at (6.5,0.5) {...};

\draw[dashed] (1,0.7) -- (1.5,0.7);
\draw[dashed] (2.5,0.7) -- (3,0.7);
\draw[dashed] (4,0.7) -- (4.4,0.7);
\draw[dashed] (5.4,0.7) -- (6,0.7);
\draw[dashed] (7,0.7) -- (7.6,0.7);
\draw[dashed] (8.6,0.7) -- (9,0.7);

\end{tikzpicture}
\end{center}

\caption{A $2$-cube visibility representation with cubes $B(c_1)$,...,$B(c_k)$ blocking $B(a)$ from $B(b)$.
The blocking cubes are in a different component from $B(a)$ and $B(b)$ when $R_1$ is partitioned.} \label{blocker}
\end{figure}
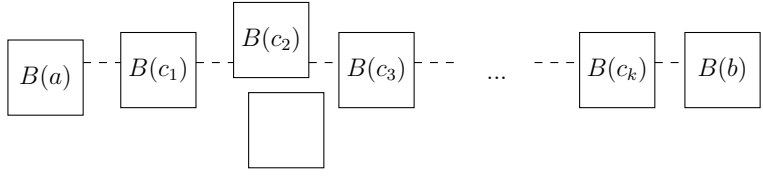

Let $P$ be the unique path in $T$ with endpoints $a$ and $b$, and let $P$ have length $m+1$.
Thus, $m\le k$.
Since $P$ is contained in the component of $B(a)$ and $B(b)$, the sum of the Euclidean lengths of the lines of sight corresponding to edges in $P$ is at least $d-m$ if $R$ is an $n$-cube visibility representation and is at least $d$ if $R$ is an $n$-cube orthogonal visibility representation.
Note that the $m+1$ edges in $P$ are also represented in $W$.
If $m<k$, then there is an edge in $P$ with a line of sight in $W$ that is shorter than its line of sight in $P$, contradicting the assumption that the shortest line of sight for each edge has been chosen.
Therefore, $m=k$, and $W=P$.
If $W$ does not contain a line of sight that is shorter than the corresponding line of sight in $P$, then the length of each line of sight in $W$ is equal to the length of the corresponding line of sight in $P$.
Therefore, in $R_1$ there are parallel lines of sight of the same length that join $B(a)$ to two $n$-cubes corresponding to the same vertex in $T$.
If $R$ is an $n$-cube visibility representation, this contradicts the assumption that no $n$-cubes corresponding to the same vertex can see each other.
If $R$ is an $n$-cube orthogonal  visibility representation, this contradicts the assumption that all cubes in $R$ have distinct heights.
We conclude that there are no lines of sight in $R'$ that correspond to edges that are not in $T$.
Thus, $R'$ is a representation of $T$ with fewer pairs of $n$-cubes that are assigned to the same vertex and lie in the same component, contradicting the minimality of $R$.
Therefore, there is a representation of $T$ that has the same number of $n$-cubes as $R$ and is a disjoint union of representations of trees, so $\Upsilon_{h^{(n)*}}(T)\le h^{(n)*}(T)$.
Thus, $h^{(n)*}(T) = \Upsilon_{h^{(n)*}}(T)$. 
\end{proof}

We now characterize $n$-cube visibility trees in terms of $(n-1)$-cube orthogonal visibility trees.
This result is a straightforward generalization of Theorem 4.5 from \cite{UnitRec}, which corresponds to the case when $n=2$.
The proof of Theorem 4.5 from~\cite{UnitRec} uses the structural characterization of $1$-cube orthogonal visibility trees due to Dean and Veytsel~\cite{UnitBar} to build $2$-cube visibility representations of trees.
Our proof does not require a characterization of unit $(n-1)$-cube orthogonal visibility trees, but does depend on unit $(n-1)$-cube orthogonal visibility representations of trees that are assumed to be unit $(n-1)$-cube orthogonal visibility trees.
In this sense, our proof is not purely constructive as is the proof of Theorem 4.5 from~\cite{UnitRec}.

We use the immediate corollary of Lemma~\ref{lem:coords} that in a unit $(n-1)$-cube orthogonal visibility representation of a graph, the heights of the cubes can be made to differ by an arbitrarily large amount.
We will also adopt the convention that the location of a cube will be given by the coordinates of the point at the center of cube (recall that the side lengths of all cubes are $1$).

\begin{theorem}\label{TreetoTree}
A tree $T$ is an $n$-cube visibility graph if and only if it is the union of $n$ $(n-1)$-cube orthogonal visibility forests.
\end{theorem}

\begin{proof}
Let $R$ be an $n$-cube visibility representation of $T$.
By Theorem~\ref{vis=arb}, we may assume that all components of $R$ are trees.
For each $i\in [n]$, the lines of sight in $R$ that are parallel to $\ee_i$ correspond to a unit $(n-1)$-cube orthogonal visibility representation of a subforest of $T$.
The union of these $n$ forests is $T$.

Now assume that $T$ is a $t$-vertex tree that is the union of $n$ $(n-1)$-cube orthogonal visibility forests $F_1,\ldots,F_n$ (we may assume that these forests are spanning, since we can add isolated cubes to their representations).
Note that distinct components of $F_j$ and $F_{j'}$ share at most one vertex.
Label the vertices of $T$ using a breadth-first search starting at an arbitrary vertex $v_0$.
For each $j$, label each subtree of $T$ that is a component of $F_j$ as $T_{i,j}$ where $i$ is the lowest index of a vertex in the tree.

We will iteratively add cubes for the vertices in the trees $T_{i,j}$ in lexicographic order until we obtain a representation of $T$.
When $T_{i,j}$ is processed, $v_i$ has already been assigned a cube.
Assume that the cube $B$ is added at $(x_1,\ldots,x_n)$ when $T_{i,j}$ is processed.
We will ensure that the following conditions are satisfied:

\begin{enumerate}
\item[(1)] When $B$ is added, it will only see the cubes of its neighbors in $T_{i,j}$, and it will see those cubes in the direction of $\ee_j$.
\item[(2)] When $B$ is added, for each $j'\neq j$, there is no cube whose location lies in the set $\{(x_1\pm 2t,\ldots,x_{j'-1}\pm 2t,z,x_{j'+1}\pm 2t,\ldots,x_n\pm 2t)|z\in \RR\}$.
\end{enumerate}

Begin by placing the cube for $v_0$ at the origin.
The component $T_{0,1}$ of $F_1$ that contains $v_0$ has an $(n-1)$-cube orthogonal visibility representation, and by Lemma~\ref{lem:coords} we can assume that all $(n-1)$-cubes in the representation have distinct heights.
Place the $n$-cubes for the vertices in $T_{0,1}$ so that (i) all lines of sight between them are parallel to $\ee_1$, (ii) they form a representation of $T_{0,1}$, and (iii) their first coordinates all differ by at least $4t$.
Condition 1 clearly holds for each cube, and condition 2 holds since the $j$th coordinates of the cubes all differ by at least $4t$.

To process $T_{i,j}$, observe that $v_i$ has been assigned an $n$-cube.
Let $B_i$ be the cube assigned to $v_i$, and assume that $B_i$ was added to the representation when processing a component of $F_{j'}$.
By Lemma~\ref{lem:coords}, $T_{i,j}$ has a unit $(n-1)$-cube orthogonal visibility representation in which all cubes have distinct heights.
Fix such a representation $R$ so that (a) the height of the cube in $R$ that is assigned to $v_i$ is the $j$th coordinate of $B_i$, (b) every other $(n-1)$-cube in the representation has a height that is either greater than or less than the $j$th coordinate of all existing $n$-cubes in the partial representation of $T$ by at least $4t$, and (c) the heights of all the $(n-1)$-cubes in the representation differ by at least $4t$.
Align such a representation so that the visibilities are parallel to $\ee_j$, and for each $v\in V(T_{i,j})-v_i$ add an $n$-cube in the location of the corresponding $(n-1)$-cube in $R$.

Let $B$ be a cube that is added when $T_{i,j}$ is processed.
Since $T_{i,j}$ is connected and has at most $t$ vertices, each coordinate of $B$ except the $j$th differs from the corresponding coordinate of $B_i$ by at most $t$.
Since all cubes that are not in $T_{i,j}$ have locations that differ from the location of $B_i$ by at least $4t$ in a direction that is not parallel to $\ee_j$, it follows that $B$ can only see the cubes of vertices in $T_{i,j}$.
Since the $j$th coordinate of $B$ also differs from the $j$th coordinate of all other cubes by at least $4t$, conditions (1) and (2) hold for $B$.

It follows that by processing all $T_{i,j}$, we obtain an $n$-cube visibility representation of $T$.
\end{proof}

Theorem~\ref{TreetoTree} allows us to relate $\Upsilon_{h^{(n)}}(T)$ to $\Upsilon_{h^{(n-1)\perp}}(T)$ for a tree $T$.

\begin{theorem}\label{Upsilons}
If $T$ is a tree, then 
$\ds\Upsilon_{h^{(n)}}(T) = \CL{\frac{\Upsilon_{h^{(n-1)\perp}}(T)}{n}}.$
\end{theorem}

\begin{proof}
Suppose that $\Upsilon_{h^{(n)}}(T)=k$, and decompose $T$ into $k$ unit $n$-cube visibility forests.
Each unit $n$-cube visibility forest is the union of $n$ unit $(n-1)$-cube orthogonal visibility forests, and therefore $\CL{\frac{\Upsilon_{h^{(n-1)\perp}}(T)}{n}} \le \Upsilon_{h^{(n)}} (T)$.

Now suppose that $\Upsilon_{h^{(n-1)\perp}}(T)=\ell$.
It follows that $T$ can be decomposed into $(n-1)$-cube orthogonal visibility trees so that each vertex is contained in at most $\ell$ of those trees.
Furthermore, by adding multiple copies of the $1$-vertex tree at each vertex, we can obtain a multiset of $(n-1)$-cube orthogonal visibility trees that are subtrees of $T$ such that (i) every vertex is contained in exactly $\ell$ members of the multiset, and (ii) the union of the trees in the multiset is $T$.
Construct an auxiliary graph $G$ in which each tree of the multiset is a vertex, and two vertices are adjacent if and only if they share a vertex in $T$.
Since $G$ is the intersection graph of subtrees of a tree, it follows that $G$ is a chordal graph \cite{Buneman, Gavril, Walter1, Walter2}.
Hence $G$ is perfect and $\chi(G)=\omega(G)$, where $\chi (G)$ is the chromatic number of $G$ and $\omega(G)$ is the size of the largest clique in $G$.
Furthermore, since the intersection of any two trees in the decomposition is at most a single vertex, $\omega(G)=\ell$.
Therefore, there is a coloring of the trees in the decomposition using $\ell$ colors so that no two trees of the same color share a vertex.
By Theorem~\ref{TreetoTree}, the union of $n$ $(n-1)$-cube orthogonal visibility forests is an $n$-cube visibility tree, so the union of $n$ color classes of the trees in the multiset is an $n$-cube visibility tree.
Thus, $T$ can be decomposed into $\CL{\ell/n}$ $n$-cube visibility trees.
Therefore, $\CL{\frac{\Upsilon_{h^{(n-1)\perp}}(T)}{n}} \ge \Upsilon_{h^{(n)}} (T)$.
\end{proof}

In~\cite{UnitBarVisNum}, Gaub et al.\ developed a fast algorithm to determine $h^{(1)\perp}(T)$ for a tree $T$ and obtain a $1$-cube orthogonal visibility representation of $T$.

\begin{theorem}[Gaub et al.~\cite{UnitBarVisNum}]~\label{thm:TreeAlg}
Let $T$ be a tree.
If $\Delta(T)\not\equiv 0\pmod 3$, then $h^{(1)\perp}=\CL{\Delta(T)/3}$.
If $\Delta(T)\equiv 0\pmod 3$, then $h^{(1)\perp}=\CL{\Delta(T)/3}$ or $h^{(1)\perp}=\CL{(\Delta(T)+1)/3}$, and there is a linear time algorithm that determines the correct value.
\end{theorem}

Using this algorithm and Theorem~\ref{Upsilons}, we can determine the $2$-cube visibility number of a tree.

\begin{theorem}\label{ur(T)}
Let $T$ be a tree. 
If $\Delta(T)\not\equiv 0\pmod 6$, then $$h^{(2)}(T) = \CL{\frac{\Delta(T)}{6}}.$$
If $\Delta(T)\equiv 0\pmod 6$, then $$h^{(2)}(T) = \CL{\frac{\Delta(T)}{6}} \text{ or } h^{(2)}(T) = \CL{\frac{\Delta(T)+1}{6}},$$
and there is a linear time algorithm to determine the exact value.
\end{theorem}

\begin{proof}
If $\Delta(T)\not \equiv 0\pmod 6$, then $\CL{\frac{\Delta(T)}{3}}\le h^{(1)\perp}(T)\le \CL{\frac{\Delta(T)+1}{3}}$, and by Lemma~\ref{Upsilons}, $h^{(2)}(T)=\CL{\frac{h^{(1)\perp}(T)}{2}}=\CL{\frac{\Delta(T)}{6}}$.
If $\Delta(T)\equiv 0\pmod 6$, then $h^{(1)\perp}(T)=\CL{\frac{\Delta(T)}{3}}$ or $h^{(1)\perp}(T)= \CL{\frac{\Delta(T)+1}{3}}$, and there is an algorithm to determine the the correct value in linear time.
Therefore, $h^{(2)}(T)=\CL{\frac{\Delta(T)}{6}}$ or $h^{(2)}(T)=\CL{\frac{\Delta(T)+1}{6}}$ and there is an algorithm to determine the correct value.
\end{proof}

The algorithm UNIT\_BAR\_TREE from~\cite{UnitBarVisNum} that determines the value of $h^{(1)\perp}(T)$ for a tree depends on the Dean-Veytsel characterization of $1$-cube orthogonal visibility trees from~\cite{UnitBar}.
Thus, similar results in the vein of Theorem~\ref{ur(T)} for higher dimensions would require and understanding the structure of $n$-cube orthogonal visibility trees for $n\ge 2$.

An {\it intersection representation} of a graph $G$ is an assignment of the vertices of $G$ to sets such that two vertices are adjacent if and only if their sets have nonempty intersection.
A graph has {\it cubicity} $n$ if it has an intersection representation where each vertex is assigned to an axis-aligned closed unit $n$-cube in $\RR^n$, but it does not have an intersection representation where each vertex is assigned to an axis-aligned closed unit $(n-1)$-cube in $\RR^{n-1}$.
Note that if a graph has cubicity $n$, then it has a representation as an intersection graph of a set of unit $m$-cubes for all $m\ge n$. 
We show that the structure of unit $n$-cube orthogonal visibility trees is closely related to the structure of trees with cubicity at most $n$.

First we address a slight discrepancy: in our visibility representations we require lines of sight to be cylindrical channels with positive diameter, and in representations as intersection graphs of cubes we may edges represented by intersections that do contain any $\epsilon$-ball (which would correspond to the diameter of the channel in the visibility graph).

\begin{lemma}\label{lem:epsilon}
If $T$ is tree with cubicity $n$, then there is a representation of $T$ as the intersection graph of $n$-cubes and an $\epsilon>0$ such that such that all nonempty intersections of the representation contain an $\epsilon$-ball.
\end{lemma}

\begin{proof}
Let $R$ be an intersection representation of $T$ with the minimum number of intersections that do not contain an $\epsilon$-ball for all $\epsilon >0$.
If there are no such intersections, then the result holds.
Assume that $u$ and $v$ are two adjacent vertices whose cubes intersect in a set that contains no $\epsilon$-ball for every $\epsilon >0$.
Let $T_u$ be the component of $T-uv$ that contains $u$, and let $T_v$ be the component that contains $v$.
There exists $\delta>0$ such that for every ordered pair of vertices $(u',v')$ where $u'\in T_u$ and $v'\in T_v$ and $(u',v')\neq (u,v)$ the cubes of $u'$ and $v'$ have locations that differ by at least $1+\delta$ in every coordinate.
Let $\mathbf{w}$ be the vector $\frac\delta 2(\sgn(v_1-u_1),\ldots,\sgn(v_n-u_n))$ where $\sgn$ denotes the sign function; translating the cube of $u$ by $\mathbf w$ will move it towards the cube of $v$ by $\frac\delta 2$ in every direction.
For each vertex $u'\in V(T_u)$, translate the cube of $u'$ by $\mathbf w$.
The only intersection of cubes affected by this translation is the intersection of the cubes of $u$ and $v$, which now contains a $\delta/4$-ball.
This contradicts the minimality of $R$. 
\end{proof}

Let $T$ be a tree and let $v\in T$.
A {\it path expansion} of length $k$ at $v$ is the process of replacing $v$ by a path $v_0,v_1,\ldots,v_k$ and partitioning the set of neighbors of $v$ so that each neighbor of $v$ is now adjacent to exactly one of $v_0$ or $v_k$.

\begin{theorem}\label{thm:pathexp}
A tree $T$ is a unit $n$-cube orthogonal visibility graph if and only if it can be obtained from a tree $T'$ with cubicity at most $n$ by performing a path expansion at each vertex of $T'$ (perhaps some with length $0$).
\end{theorem}

\begin{proof}
Let $T$ be obtained from the tree $T'$ with cubicity at most $n$ by performing a path expansion at each vertex of $T'$.
Consider an $n$-cube intersection representation of $T'$ in which all nonempty intersections contain an $\epsilon$-ball for some $\epsilon>0$, which is guaranteed by Lemma~\ref{lem:epsilon}.
Because $T'$ contains no cycles, it follows that each point in $\RR^n$ is contained in the $n$-cubes of at most two vertices in $T'$.
Place the representation of $T'$ into $\RR^{n+1}$ so that the cubes occupy the subspace spanned by $\{\ee_1,\ldots,\ee_n\}$.

Fix a vertex $v$ in $T'$.
Starting from $v$ we 2-color the edges of $T'$ using a breadth-first search in the following fashion.
When performing the path-expansion at $v$ to obtain $T$, the neighbors of $v$ in $T'$ are partitioned into two sets $N_1(v)$ and $N_2(v)$.
Color the edges joining $v$ and vertices in $N_1(v)$ blue, and color the edges joining $v$ and vertices in $N_2(v)$ red.
When processing the edges at a vertex $u\neq v$, exactly one edge $uu'$ is already colored, where $u'$ is the predecessor of $u$ in the breadth first search.
Suppose that $u'\in N_i(u)$ when the path expansion at $u$ is performed.
Color all edges joining $u$ to vertices in $N_i(u)$ the same color as $uu'$.
Color all other edges at $u$ with the other color.

Now place the $n$-cube for $v$ in the subspace of $\RR^{n+1}$ spanned by $\{\ee_1,\ldots,\ee_n\}$ (so that the value of the $(n+1)$st coordinate is $0$).
For $u\in V(T)$, let $b_{vu}$ denote the number of blue edges on the path from $v$ to $u$, and let $r_{vu}$ denote the number of red edges on the path from $v$ to $u$.
For each $u$ in $v$, let the $(n+1)$st coordinate of the corresponding $n$-cube in $\RR^{n+1}$ be given by $b_{vu}-r_{vu}$.
This is a unit $n$-cube orthogonal visibility representation of $T'$.

Let $v$ be a vertex in $T'$ where a path expansion is performed that produces a path of length $k$.
To construct a unit $n$-cube orthogonal visibility representation of $T$, place a stack of $k$ evenly spaced unit $n$-cubes at heights between $b_{vu}-r_{vu}-\frac 13$ and $b_{vu}-r_{vu}+\frac 13$.

Now suppose that $T$ has an orthogonal $n$-cube representation.
Project the cubes onto the subspace of $\RR^{n+1}$ spanned by $\{\ee_1,\ldots,\ee_n\}$.
The result is an intersection representation using unit $n$-cubes of a tree $T'$ with cubicity at most $n$.
The cubes that project onto the same cube form a path in $T$, and this process can be reversed using path expansions.
\end{proof}

Trees with cubicity $1$ are clearly paths, so the characterization of unit bar visibility graphs is a corollary to Theorem~\ref{thm:pathexp}.

\begin{theorem}\label{DV}[Dean and Veytsel~\cite{UnitBar}]
A tree is a unit bar visibility graph if and only if it is a subdivided caterpillar with maximum degree $3$.
\end{theorem}

At the moment, results similar to Thoerem~\ref{DV} appear out of reach since the complexity of determining even if a tree has cubicity $2$ is still unknown~\cite{BBCRS}.
However, we are able to give very strong bounds on both the $n$-cube visibility numbers and $n$-cube orthogonal visibility numbers of trees that are equal in many cases.

\begin{theorem}
If $T$ is a tree, then
$$\CL{\frac{\Delta(T)}{2^n+1}}\le h^{(n)\perp}(T)\le \CL{\frac{\Delta(T)+1}{2^n+1}}$$
and
$$\CL{\frac{\Delta(T)}{n(2^{n-1}+1)}}\le h^{(n)}(T)\le \CL{\frac{\Delta(T)+1}{n(2^{n-1}+1)}}.$$
\end{theorem}

\begin{proof}
First note that if a tree $T$ has cubicity $n$, then $\Delta(T)\le 2^n$.
Therefore, by Theorem~\ref{thm:pathexp}, if $T$ is an $n$-cube orthogonal visibility graph, then $\Delta(T)\le 2^n+1$ (obtained by performing a path expansion in which the partition of the neighbors of $v$ has only one nonempty set).
Therefore, a decomposition of a tree $T$ into $n$-cube orthogonal visibility forests requires at least $\CL{\frac{\Delta(T)}{2^n+1}}$ forests.

By Theorem~\ref{TreetoTree}, if $T$ is a $n$-cube visibility graph, then $\Delta(T)\le n(2^{n-1}+1)$.
Therefore, a decomposition of a tree $T$ into $n$-cube orthogonal visibility forests requires at least $\CL{\frac{\Delta(T)}{n(2^{n-1}+1)}}$ forests.

Also, note that $K_{1,2^n}$ has cubicity $n$.
It then follows from Theorem~\ref{thm:pathexp} that $K_{1,2^n+1}$ is an $n$-cube orthogonal visibility graph, and from Theorem~\ref{TreetoTree} that $K_{1,n(2^{n-1}+1)}$ is an $n$-cube visibility graph.
A tree $T$ can be greedily decomposed into $\CL{\frac{\Delta(T)+1}{k}}$ forests with maximum degree $k$ in which every component is a star.
\end{proof}

\section{Conclusion}

We conclude with some open problems and directions for further research.

\begin{question}
For $n\ge 2$, is there a simple characterization of trees with $n$-cube orthogonal visibility representations?
\end{question}

Theorem~\ref{thm:pathexp} suggests the following more fundamental question, which is not being posed for the first time here (see~\cite{BBCRS} and references therein).

\begin{question}
For $n\ge 2$, is there a simple characterization of graphs with cubicity $n$?
\end{question}

Of course one can also consider visibility numbers of graphs that are not trees.
In his master's thesis~\cite{Peterson}, the first author studied rectangle and unit rectangle visibility numbers of a variety of graphs including complete graphs and complete bipartite graphs.
The following is perhaps the most natural starting point for a systematic study of visibility numbers in higher dimensions.

\begin{question}
How do $h^{(n)}(K_t)$ and $h^{(n)\perp}(K_t)$ grow as functions of $n$ and $t$?
\end{question}

\end{document}